\documentclass[12pt,a4paper]{amsart}
\usepackage{amsmath,amsfonts,amsthm,amsopn,color,amssymb,enumitem}
\usepackage[a4paper]{geometry}
\usepackage{palatino}
\usepackage{graphicx}
\usepackage[colorlinks=true]{hyperref}
\hypersetup{urlcolor=blue, citecolor=red, linkcolor=blue}

\usepackage{cite}
\usepackage{relsize}
\usepackage{esint}
\usepackage{verbatim}
\usepackage{mathrsfs}
\usepackage{xcolor}
\usepackage{tikz}

\newcommand{\R}{\mathbb{R}}

\newcommand{\RN}{{\mathbb{R}^N}}

\newcommand{\beq}{\begin{equation}}
\newcommand{\eeq}{\end{equation}}

\newtheorem{theorem}{Theorem}[section]
\newtheorem*{theorem*}{Theorem}
\newtheorem{lemma}[theorem]{Lemma}

\newtheorem{proposition}[theorem]{Proposition}

\theoremstyle{definition}
\newtheorem{remark}[theorem]{Remark}

\begin{document}

\title[The Neumann Green function of the annulus]{The Neumann Green function of the annulus}

\author[G. M. Rago]{Giuseppe Mario Rago}
\address{\noindent  Dipartimento di Matematica, Universit\'a degli Studi di Bari Aldo Moro,Italy }
\email{g.rago6@phd.uniba.it}

%\thanks{\textcolor{red}{Work partially supported by 
%the MUR-PRIN-P2022YFAJH ``Linear and Nonlinear PDE’s: New directions and Applications" and by the INdAM-GNAMPA project ``Fenomeni non lineari: problemi locali e non locali e loro applicazioni" CUP E5324001950001.??}}


\maketitle
\begin{abstract}
Using Gegenbauer polynomials and the zonal harmonic functions we build an explicit representation formula for the Green function with Neumann boundary conditions in the annulus.
\end{abstract}

\section{Introduction and main results}
The Green function of the operator $-\Delta$ with Neumann boundary conditions is defined by 
\begin{equation*}
\begin{cases}
- \Delta_{x} G(x, y) = \delta_{y}(x) \ \ \ \ \ \ \mbox{in} \ \Omega, \ \\
\frac{\partial G(x, y)}{\partial \nu} = - \frac{1}{|\partial \Omega|}\ \ \ \ \ \ \ \ \ \ \ \ \ \ \mbox{on}\ \partial \Omega, \\
\int_{\Omega} G(x, y) \ dx = 0
\end{cases}
\end{equation*} \\ 
where $\delta_{y}$ is the Dirac function centered at $y$ and $\Omega$ is a bounded domain of $\mathbb{R}^{N}$ with $N \geq 3$. 
It is well known that the Green function can be written as 
\begin{equation*}
G(x,y) = \frac{1}{\omega_{N-1}(N-2)|x-y|^{N-2}} - H(x,y)
\end{equation*}
where $H(x,y)$ is a smooth function in $\Omega \times \Omega$ which is harmonic in both variables $x$ and $y$. Finally, the Robin function is defined by 
\begin{equation*}
\tau(x) := H(x,x).
\end{equation*}
The knowledge of the Green (or the Robin) function is of great importance in applications (\cite{BF}). Indeed, the explicit calculation of the Green function is an old problem (see for example the book by Courant and Hilbert, \cite{CH}) but it can be solved only in special cases (like the ball or half-space).
For these reasons, even if it is not possible to have the explicit expression, it is very important to deduce any properties of the Green function. \ \\
In this paper we are interested in the case where the domain is the annulus in $\mathbb{R}^{N}$, namely
\begin{equation*}
\Omega_{a,b} : = \{ x \in \mathbb{R}^{N} : a < |x| < b\},
\end{equation*}
where $a,b \in \mathbb{R}$. For simplicity, in computations, we will assume that $b=1$ and hence $\Omega_{a,1} : = \Omega_a$. \ \\
However, the goal of this paper is to give an explicit representation formula for the Green function with Neumann boundary conditions when $N \geq 3$, employing the zonal spherical harmonics, which have already been used to derive some important results on the existence of solutions for some classes of elliptic problems in annular domains (see for instance \cite{MRV}, \cite{MS}, \cite{PRV}). \ \\
The result is the following. 
\begin{theorem}\label{principale}
Let $\Omega_a$ be the annulus 
\begin{equation}\label{anello}
\Omega_{a} : = \{ x \in \mathbb{R}^{N} : a < |x| < 1\},
\end{equation}
where $0<a<1$. Then we have that the Green function in $\Omega_{a}$ is given by 
\begin{equation*}
G_{a}(x,y) = \frac{1}{\omega_{N-1}(N-2) |x-y|^{N-2}} - H_{a}(x, y)
\end{equation*}
where
\begin{equation}\label{ansatzregularpart}
H_{a}(x, y) = \frac{1}{\omega_{N-1}} \Bigg[ \sum_{m=1}^{+\infty} A_{m}(|x|) |y|^{m} + \sum_{m=1}^{+\infty} B_{m}(|x|) |y|^{-(N+m-2)} \Bigg] Z_{m} \left( \frac{x}{|x|}, \frac{y}{|y|} \right) + C_{0} |y|^{-(N-2)}
\end{equation}
with
\begin{equation}\label{Am}
A_{m}(|x|) = \frac{m+N-2}{m(2m+N-2)} \frac{|x|^{m}}{a^{2m+N-2}-1} \Bigg[ 1 +\frac{m}{m+N-2} \Bigg( \frac{a}{|x|} \Bigg)^{2m+N-2} \Bigg],
\end{equation}
\begin{equation}\label{Bm}
B_{m}(|x|)= \frac{a^{2m+N-2}}{2m+N-2} \frac{|x|^{m}}{a^{2m+N-2}-1} \Bigg[ 1 + \frac{m}{m+N-2} \left( \frac{1}{|x|} \right)^{2m+N-2} \Bigg]
\end{equation}
and 
\begin{equation}\label{C0}
C_0 = \frac{1}{(N-2)\omega_{N-1}} \frac{a^{N-1}}{1+a^{N-1}};
\end{equation}
\end{theorem}
\begin{remark}
We can observe that in the representation formula of the Green function doesn't appear a constant as usual because this constant is identically equal to \textit{zero} because of the zero-average condition of the Green function on the annulus (see for instance formula $(45)$ of \cite{F}).
\end{remark}
\noindent We recall that $Z_{m} \left( \frac{x}{|x|}, \frac{y}{|y|} \right)$ are the zonal harmonics function associated with spherical harmonics of degree $m$. They are a key ingredient in the representation formula of the Green function and whose main properties will be recalled and exploited in the subsequent sections. \ \\
Furthermore, the representation formula derived in this paper is closely related, in spirit, to the explicit Green function obtained by Grossi and Vujadinović for annular domains with Dirichlet boundary conditions in \cite{GV}. In both settings, the use of zonal spherical harmonics and Gegenbauer polynomials allows for a precise description of the singular and regular parts of the Green function.\ \\
Nevertheless, the passage from Dirichlet to Neumann boundary conditions is not straightforward. \ \\
Indeed, in the Neumann case, the Laplace operator is not coercive and the Green function is uniquely determined only up to additive constants, which makes it necessary to impose a zero-average normalization. This feature produces additional terms in the representation formula and yields a reduced problem with a different structure compared to the Dirichlet case, where the Robin function alone governs the leading-order behavior.
As a consequence, although the two constructions share a common analytical framework, the Neumann Green function exhibits genuinely new features that are essential in applications to critical problems with Neumann boundary conditions. \ \\ \\
The paper is organized as follows: in Section $2$ we recall some useful preliminaries about zonal spherical harmonics and the Gegenbauer polynomials and in Section $3$ we prove Theorem \ref{principale}.

\section{Preliminaries}
In this section we would like to point out the basic properties of zonal harmonics which are going to be used through the paper. A good reference for the interested reader is the book \cite{A}. \ \\
By $H_{m}(\mathbb{R}^{N})$ we are going to denote the finite dimensional Hilbert space of all harmonic homogeneous polynomials of degree $m$. \ \\
Let us denote by $\mathbb{S}_{N-1}$ the unite sphere of $\mathbb{R}^{N}$. A \textit{spherical harmonic} of degree $m$ is the restriction to $\mathbb{S}_{N-1}$ of an element of $H_{m}(\R^{N})$. The collection of all spherical harmonics of degree $m$ will be denoted by $H_{m}(\mathbb{S}_{N-1})$. \ \\
Now we consider an important subset of $H_{m}(\mathbb{S}_{N-1})$, the so-called \textit{zonal harmonics}. They can be defined in different ways. One of these is the definition given in Theorem $5.38$ in \cite{A}. \ \\
For $x \in \RN$ with $N \geq 2$ and $\xi \in \mathbb{S}_{N-1}$ we define the zonal harmonic $Z_{m}(x, \xi)$ of degree $m$ as 
\begin{equation*}
Z_{m}(x, \xi) = (N+2m-2) \sum_{k=0}^{\lceil \frac{m}{2} \rceil} (-1)^{k} \frac{N(N+2) \cdots (N+2m-2k-4)}{2^{k} k! (m-2k)!}(x \cdot \xi)^{m-2k} |x|^{2k}
\end{equation*}
for all $m \geq 1$ and $Z_{0}(x, \xi)=1$. \ \\
Several properties of the zonal harmonics can be found in Chapter $5$ of \cite{A}. \ \\
Another important aspect of the zonal harmonics is their relation with ultraspherical polynomials. 
Indeed the zonal harmonics have a particularly simple expression in terms of Gegenbauer (or ultraspherical) polynomials $C_{m}^{\lambda}$. The latter can be defined in terms of generating functions. If we write (see \cite{SW} p.148)
\begin{equation}\label{sviluppo}
(1-2rt+r^2)^{-\lambda} =  \sum_{m=0}^{+\infty} C_{m}^{\lambda}(t) r^{m},
\end{equation}
where $0 \leq |r| < 1$, $|t| \leq 1$ and $\lambda>0$, then the coefficient $C_{m}^{\lambda}$ is called Gegenbauer polynomial of degree $m$ associated with $\lambda$.
The next result (see Theorem $2.1$ \cite{GV}) is related to representation of the zonal harmonics.
\begin{theorem}
If $N>2$ is an integer, $\lambda= \frac{N-2}{2}$ and $m=0,1,2, \cdots$ then we have that for all $x',y' \in \mathbb{S}_{N-1}$, it holds 
\begin{equation*}
Z_{m}(x',y') = \frac{2m+N-2}{N-2} C_{m}^{\lambda} (x' \cdot y').
\end{equation*}
\end{theorem}
\begin{remark}
In the following we will use these notations: $|x|=\rho$ and $|y|=r$. 
\end{remark}
\section{The representation formula for the Green function}
First of all, we consider the annulus $\Omega_{a}$ as in \eqref{anello} and we consider the Neumann Green function 
\begin{equation*}
\begin{cases}
\displaystyle - \Delta_{x} G_{a}(x, y) = \delta_{y}(x) \ \ \ \ \ \ \ \ \ \ \ \ \ \ \ \mbox{in} \ \Omega_a, \ \\
\displaystyle \frac{\partial G_{a}(x, y)}{\partial \nu} = - \frac{1}{|\partial \Omega_a|} \ \ \ \ \ \ \ \ \ \ \ \ \ \ \ \ \mbox{on}\ \partial \Omega_a, \\
\displaystyle \int_{\Omega} G_{a}(x, y) \ dx = 0
\end{cases}
\end{equation*} \\ 
where $|\partial \Omega_a| = \omega_{N-1} (1 + a^{N-1})$ and $\omega_{N-1}$ is the measure of the surface area of the sphere in $\mathbb{R}^{N-1}$. \ \\
The regular part
\begin{equation}\label{parteregepartesing}
H_{a}(x, y ) : = \Gamma(y-x) - G_{a}(x, y)
\end{equation}
satisfies the following problem
\begin{equation}\label{robin}
\begin{cases}
\displaystyle - \Delta_{x} H(x, y) = 0 \ \ \ \ \ \ \ \ \ \ \ \ \ \ \ \ \ \ \ \ \ \ \ \ \ \ \ \ \ \ \ \ \mbox{in} \ \Omega_a, \ \\
\displaystyle  \frac{\partial H(x, y)}{\partial \nu} = \frac{\partial \Gamma(x-y)}{\partial \nu} + \frac{1}{|\partial \Omega_a|}\ \ \ \ \ \ \mbox{on}\ \partial \Omega_a, \\
\end{cases}
\end{equation}
Our first aim is to write the Green function for the annulus in terms of the zonal spherical harmonics $Z_{m} \left( \frac{x}{|x|}, \frac{y}{|y|} \right)$. The starting point for our results is going to be the next lemma which will play an important role in proving Theorem \ref{principale}.
\begin{lemma}\label{espansione}
We have that, for any $|x|, |y| \leq 1$ such that $|y|<|x|$ and $y \neq x$,
\begin{equation}\label{prima}
\frac{1}{|x-y|^{N-2}} = \sum_{m=0}^{+\infty} \frac{N-2}{2m+N-2} \frac{|y|^{m}}{|x|^{m+N-2}} Z_{m} \left( \frac{x}{|x|}, \frac{y}{|y|} \right).
\end{equation}
Instead, for any $|x|, |y| \leq 1$ such that $|x|<|y|$ and $y \neq x$,
\begin{equation}\label{seconda}
\frac{1}{|x-y|^{N-2}} = \sum_{m=0}^{+\infty} \frac{N-2}{2m+N-2} \frac{|x|^{m}}{|y|^{m+N-2}} Z_{m} \left( \frac{x}{|x|}, \frac{y}{|y|} \right).
\end{equation}
\end{lemma}
\begin{proof}
We will prove only \eqref{prima} because with a similar computation it can be proved also the \eqref{seconda}. \ \\
%\textcolor{red}{Così dovrebbe andare bene, poiché il motivo per il quale è importante evidenziare che nei due diversi casi si ha un'espansione diversa è legato al fatto che la $s$ per come è stata definita successivamente deve essere sempre minore di 1 per avere la convergenza per avere l'analiticità della funzione di variabile complessa $s$ e la successiva sviluppabilità in serie ricordandop chiaramente che quella serie converge assolutamente e uniformemente sui compatti se $s<1$.} \ \\
Let us consider $\lambda=\frac{N-2}{2}$,
\begin{equation*}
t := \frac{x}{|x|} \cdot \frac{y}{|y|} \in [-1,1]
\end{equation*}
and $r:=\frac{|y|}{|x|} < 1$.
Now, by using formula \eqref{sviluppo}, we obtain that
\begin{equation*}
\begin{aligned}
\frac{1}{|x-y|^{N-2}} & = |x|^{-(N-2)} \Bigg( 1-2 \frac{|y|}{|x|} \left( \frac{y}{|y|} \cdot \frac{x}{|x|} \right) + \frac{|y|^{2}}{|x|^{2}} \Bigg)^{-\lambda} =  |x|^{-(N-2)} (1-2rt+r^{2})^{-\lambda}
\\& =  |x|^{-(N-2)} \sum_{m=0}^{+\infty} C_{m}^{\lambda}(t) r^{m} = \sum_{m=0}^{+\infty} \frac{N-2}{2m+N-2} \frac{|y|^{m}}{|x|^{m+N-2}} Z_{m} \left(\frac{x}{|x|}, \frac{y}{|y|}\right).
\end{aligned}
\end{equation*}
Furthermore, if we exchange the role of $x$ and $y$ we obtain the \eqref{seconda} and we get the claim. Moreover, the series in \eqref{prima} and \eqref{seconda} converges absolutely and uniformly on compact subsets of the region $|y|<|x|$ (respectively $|x|<|y|$).
\end{proof}
The following result exploits the two different expressions of the radial derivative of the singular part of the Green function. \ \\

\begin{proposition}\label{y<xi}
We have that, for any $|x|, |y| \leq 1$ such that $|y|<|x|$ and $y \neq x$,
\begin{equation}\label{y<xi}
\partial_{r}\Gamma(y-x) = \frac{1}{\omega_{N-1}} \sum_{m=0}^{+\infty} \frac{m}{2m+N-2} \frac{|y|^{m-1}}{|x|^{m+N-2}} Z_{m} \left( \frac{x}{|x|}, \frac{y}{|y|} \right).
\end{equation}
Instead, for any $|x|, |y| \leq 1$ such that $|x|<|y|$ and $y \neq x$,
\begin{equation}\label{xi<y}
\partial_{r}\Gamma(y-x) = \frac{1}{\omega_{N-1}} \sum_{m=0}^{+\infty} -\frac{m+N-2}{2m+N-2} \frac{|x|^{m}}{|y|^{m+N-1}} Z_{m} \left( \frac{x}{|x|}, \frac{y}{|y|} \right).
\end{equation}
\end{proposition}

\begin{proof} 
In order to prove \eqref{y<xi} we observe that
\begin{equation*}
\begin{aligned}
\partial_{r} \Gamma(y-x) & = \frac{1}{\omega_{N-1}(N-2)} \partial_{r} \frac{1}{|x-y|^{N-2}} 
\\& =  \frac{1}{\omega_{N-1}(N-2)} \frac{\partial}{\partial r} \Bigg[ \sum_{m=0}^{+\infty} \frac{N-2}{2m+N-2} \frac{|y|^{m}}{|x|^{m+N-2}} Z_{m} \left( \frac{x}{|x|}, \frac{y}{|y|} \right) \Bigg] 
\\& =  \frac{1}{\omega_{N-1}} \sum_{m=0}^{+\infty} \frac{m}{2m+N-2} \frac{|y|^{m-1}}{|x|^{m+N-2}} Z_{m} \left( \frac{x}{|x|}, \frac{y}{|y|} \right)
\end{aligned}
\end{equation*}
and this concludes the proof. \ \\
While, in order to obtain \eqref{xi<y}, we observe that
\begin{equation*}
\begin{aligned}
\partial_{r} \Gamma(y-x) & = \frac{1}{\omega_{N-1}(N-2)} \partial_{r} \frac{1}{|x-y|^{N-2}} 
\\&=  \frac{1}{\omega_{N-1}(N-2)} \frac{\partial}{\partial r} \Bigg[ \sum_{m=0}^{+\infty} \frac{N-2}{2m+N-2} \frac{|x|^{m}}{|y|^{m+N-2}} Z_{m} \left( \frac{x}{|x|}, \frac{y}{|y|} \right) \Bigg] 
\\& =  \frac{1}{\omega_{N-1}} \sum_{m=0}^{+\infty} - \frac{m+N-2}{2m+N-2}  \frac{|x|^{m}}{|y|^{m+N-1}} Z_{m} \left( \frac{x}{|x|}, \frac{y}{|y|} \right)
\end{aligned}
\end{equation*}
and this concludes the proof. \ \\
Moreover, by Lemma \ref{espansione}, the series representation of the Green kernel converges uniformly on compact subsets of region $|y|<|x|$. Moreover, the corresponding series of radial derivatives converges uniformly as well, thus allowing term-by-term differentiation.
\end{proof}

Now, we have all the tools that we need for proving Theorem \ref{principale}.

\begin{proof}[Proof of Theorem \ref{principale}.]
Firstly, we point out that the ansatz of the regular part in \eqref{ansatzregularpart} follows from the general expansion of harmonic functions in annular domain in terms of spherical harmonics (see for instance formula $(2.4)$ of \cite{L}). \ \\
Therefore, in order to conclude the proof, we need to understand what the constants $A_{m}(|x|)$, $B_{m}(|x|)$ and $C_{0}$ are. \ \\
We recall that the regular part satisfies \eqref{robin} and, in order to find the right constants, we have to use the Neumann boundary conditions. \ \\
First of all, we can observe that on the two components of the boundary the external normal $\nu$ has opposite orientations with respect to the radial derivative:
\begin{itemize}
\item on the outer sphere, for $r=1$ : $\partial_{\nu} = \partial_{r}$;
\item on the inner sphere, for $r=a$ : $\partial_{\nu} = - \partial_{r}$.
\end{itemize}
Therefore, we can rewrite \eqref{robin} as
\begin{equation}\label{first}
\frac{\partial H(x, y)}{\partial r}_{\Big |_{r=1}} = \frac{\partial \Gamma(x-y)}{\partial r}_{\Big |_{r=1}}  + \frac{1}{|\partial \Omega_a|}
\end{equation}
and
\begin{equation}\label{second}
\frac{\partial H(x, y)}{\partial r}_{\Big |_{r=a}} = \frac{\partial \Gamma(x-y)}{\partial r}_{\Big |_{r=a}}  - \frac{1}{|\partial \Omega_a|}.
\end{equation}
Now, using \eqref{y<xi} and \eqref{xi<y}, we can evaluate $\partial_{r} \Gamma(x-y)$ for $r=a$ and for $r=1$. \ \\
In particular, for $r=a$, using \eqref{y<xi}, we have that
\begin{equation*}
\begin{aligned}
\frac{\partial \Gamma(x-y)}{\partial r}_{\Big |_{r=a}} & =\frac{1}{\omega_{N-1}} \sum_{m=0}^{+\infty} \frac{1}{2m+N-2} \frac{ma^{m-1}}{|x|^{m+N-2}} Z_{m} \left( \frac{x}{|x|}, \frac{y}{|y|} \right)
\\& = \frac{1}{\omega_{N-1}} \sum_{m=1}^{+\infty} \frac{1}{2m+N-2} \frac{ma^{m-1}}{|x|^{m+N-2}} Z_{m} \left( \frac{x}{|x|}, \frac{y}{|y|} \right)
\end{aligned}
\end{equation*}
because $Z_{0} \left( \frac{y}{|y|}, \frac{x}{|x|} \right) = 1$. \ \\
Instead, using \eqref{xi<y}, we have that
\begin{equation*}
\begin{aligned}
\frac{\partial \Gamma(x-y)}{\partial r}_{\Big |_{r=1}} & = \frac{1}{\omega_{N-1}} \sum_{m=0}^{+\infty}-\frac{m+N-2}{2m+N-2} |x|^{m} Z_{m} \left( \frac{x}{|x|}, \frac{y}{|y|} \right)
\\& = - \frac{1}{\omega_{N-1}} + \sum_{m=1}^{+\infty} -\frac{m+N-2}{2m+N-2} |x|^{m} Z_{m} \left( \frac{x}{|x|}, \frac{y}{|y|} \right).
\end{aligned}
\end{equation*}
The next step is to compute the normal derivative of the regular part. \ \\
It is easy to verify that
\begin{equation*}
\begin{aligned}
\partial_{r} H(x, y) & = \frac{1}{\omega_{N-1}} \sum_{m=1}^{+\infty} \left( m A_{m}(\rho) r^{m-1} - (m+N-2) B_{m} (\rho) r^{-(m+N-1)} \right) Z_{m} \left( \frac{x}{|x|}, \frac{y}{|y|} \right) 
\\&+ (2-N) C_{0} r^{1-N}.
\end{aligned}
\end{equation*}
Then, in particular, we have 
\begin{equation*}
\begin{aligned}
\partial_{r} H(x, y)_{\Big|_{r=a}} & = \frac{1}{\omega_{N-1}} \sum_{m=1}^{+\infty} \left(m a^{m-1} A_{m}(\rho) - (m+N-2) a^{-(m+N-1)}  B_{m} (\rho) \right) Z_{m} \left( \frac{x}{|x|}, \frac{y}{|y|} \right) 
\\& + (2-N) C_{0} a^{1-N}
\end{aligned}
\end{equation*}
and
\begin{equation*}
\begin{aligned}
\partial_{r} H(x, y)_{\Big|_{r=1}} & = \frac{1}{\omega_{N-1}} \sum_{m=1}^{+\infty} \left(m A_{m}(\rho) - (m+N-2) B_{m} (\rho) \right) Z_{m} \left( \frac{x}{|x|}, \frac{y}{|y|} \right) 
\\& + (2-N) C_{0}.
\end{aligned}
\end{equation*}
Now, for all $m \geq 1$, we need that the individual contributions given by the radial derivative of the function $\Gamma(y-x)$ and the regular part $H(y,x)$ balance each other and cancel out in the case $r=a$ and in the case $r=1$. \ \\
Now, if we impose these conditions, we obtain the following system
\begin{equation*}
\begin{cases}
\displaystyle m A_{m}(\rho) - (m+N-2) B_{m}(\rho) = - \frac{m+N-2}{2m+N-2} \rho^{m}\ \\
\displaystyle m a^{m-1} A_{m}(\rho) - (m+N-2) a^{-(m+N-1)} B_{m}(\rho) = \frac{m a^{m-1}}{2m+N-2} \rho^{-(m+N-2)}
\end{cases}
\end{equation*} 
in which we have to find $A_{m}(\rho)$ and $B_{m}(\rho)$. \ \\
First of all, let us consider
$$
D =\begin{bmatrix}
\displaystyle m & -(m+N-2) \\
\displaystyle ma^{m-1} & -(m+N-2)a^{-(m+N-1)}
\end{bmatrix}
$$
$$
D_{A_{m}(\rho)} =
\begin{bmatrix}
\displaystyle -\dfrac{m+N-2}{2m+N-2} \rho^{m} & -(m+N-2) \\
\displaystyle \dfrac{m a^{m-1}}{2m+N-2} \rho^{-(m+N-2)} & -(m+N-2)a^{-(m+N-1)}
\end{bmatrix}
$$
$$
D_{B_{m}(\rho)} =\begin{bmatrix}
\displaystyle m & -\dfrac{m+N-2}{2m+N-2} \rho^{m} \\
\displaystyle ma^{m-1} & \dfrac{m a^{m-1}}{2m+N-2} \rho^{-(m+N-2)}.
\end{bmatrix}
$$
We can observe that 
\begin{equation*}
\mbox{det} D = m (m+N-2) (a^{m-1} - a^{-(N+m-1)}) \neq 0.
\end{equation*}
Furthermore,
\begin{equation*}
\mbox{det} D_{A_{m}(\rho)} = \frac{(m+N-2)^{2}}{2m+N-2} \rho^{m} a^{-(m+N-1)} + \frac{m (m+N-2)}{2m+N-2} a^{m-1} \rho^{-(m+N-2)}
\end{equation*}
and
\begin{equation*}
\mbox{det} D_{B_{m}(\rho)} =  \frac{m^2 a^{m-1}}{2m+N-2} \rho^{-(m+N-2)} + \frac{m (m+N-2)}{2m+N-2} a^{m-1} \rho^{m}.
\end{equation*}
Now, we are able to find 
\begin{equation*}
A_{m}(\rho) = \frac{\mbox{det} D_{A_{m}(\rho)}}{\mbox{det} D} 
\end{equation*}
and 
\begin{equation*}
B_{m}(\rho) =  \frac{\mbox{det} D_{B_{m}(\rho)}}{\mbox{det} D}.
\end{equation*}
Indeed, we have that 
\begin{equation*}
\begin{aligned}
A_{m}(|x|) & = \frac{1}{m(2m+N-2)} \frac{1}{a^{m-1}-a^{-(m+N-1)}} [(m+N-2) a^{-(m+N-1)} |x|^{m} + ma^{m-1} |x|^{-(m+N-2)}]
\\& = \frac{1}{m(2m+N-2)} \frac{1}{a^{2m+N-2}-1} [(m+N-2) |x|^{m} + m a^{2m+N-2} |x|^{-(m+N-2)}]
\\& = \frac{m+N-2}{m(2m+N-2)} \frac{|x|^{m}}{a^{2m+N-2}-1} \Bigg[ 1 +\frac{m}{m+N-2} \Bigg( \frac{a}{|x|} \Bigg)^{2m+N-2} \Bigg];
\end{aligned}
\end{equation*}
while 
\begin{equation*}
\begin{aligned}
B_{m}(|x|) & = \frac{a^{m-1}}{(2m+N-2)(m+N-2)} \frac{1}{a^{-(m+N-1)}} \frac{1}{a^{2m+N-2}-1} \Bigg[ m |x|^{-(m+N-2)} + (m+N-2) |x|^{m} \Bigg] 
\\& = \frac{a^{2m+N-2}}{2m+N-2} \frac{|x|^{m}}{a^{2m+N-2}-1} \Bigg[ 1 + \frac{m}{m+N-2} \left( \frac{1}{|x|} \right)^{2m+N-2} \Bigg].
\end{aligned}
\end{equation*}
%\textcolor{red}{Ora bisogna sfruttare la condizione al bordo per determinare il coeffieciente della parte puramente radiale che deriva dal caso $m=0$, fatto quello si è concluso. Prima di procedere in tal senso resta da capire il come giustificare il fatto che vi è una diversa espansione del termine $|x-y|^{2-N}$ a seconda che $|x|<|y|$ e viceversa e bisogna anche scrivere un paragrafo completa in cui si mettano in luce le caratteristiche del laplaciano sferico e soprattutto la questione dei suoi autovalori sempre ragionando con il grado $m$.} \ \\
Finally, in order to estimate the remaining terms, we can rewrite \eqref{first} as
\begin{equation*}
(2-N) C_{0} = -\frac{1}{\omega_{N-1}} + \frac{1}{\omega_{N-1}(1+a^{N-1})}
\end{equation*}
and we can rewrite \eqref{second} as 
\begin{equation*}
(2-N) C_{0} a^{1-N} =- \frac{1}{\omega_{N-1}(1+a^{N-1})}.
\end{equation*}
Therefore, with a straightforward computation, we can observe that
\begin{equation*}
C_0 = \frac{1}{(N-2)\omega_{N-1}} \frac{a^{N-1}}{1+a^{N-1}}.
\end{equation*}
\end{proof}

\begin{remark}
In \cite{S}, the main contribution concerning the annular domain lies in the detailed analysis of the Neumann Green function carried out in Section $7$. There, the radial symmetry of the annulus is exploited to obtain a representation of the Green function that allows for a precise control of its regular part and of the interaction terms arising in the Lyapunov–Schmidt reduction. This representation plays a crucial role in the study of multi-bubble configurations, as it makes it possible to express the interaction among concentrating peaks in terms of a finite-dimensional matrix built from the Green function and the associated Robin function. The use of zonal harmonics is therefore instrumental in capturing the symmetry of the domain and in deriving accurate estimates, rather than in producing an explicit closed-form expression. \ \\
Finally, we point out that our result obtained in dimension $N \geq 3$ is perfectly coherent with what Salazar obtained in the case $N=3$ (see Section $7$ of \cite{S}).
\end{remark}

\begin{remark}
In a similar spirit, the present analysis can be used to investigate a four-dimensional analogue of the problem studied by Salazar in dimension $N=3$. In Salazar’s work, the construction of sign-changing multi-bubble solutions for the critical Neumann problem is based on a finite-dimensional reduction in which the interaction among concentrating peaks is encoded in a matrix whose entries are expressed in terms of the Green function and the associated Robin function. In the case of the annulus, the explicit control of the Neumann Green function plays a fundamental role, as it allows one to verify the non-degeneracy conditions required by the reduction and to identify symmetric configurations leading to multi-bubble solutions. \ \\
In dimension $N=4$, an analogous reduction scheme is available through the results of Pistoia, Rago and Vaira (\cite{PRV}), which provide a precise description of the interaction among multiple bubbles for critical elliptic problems and identify the corresponding reduced energy expansions in terms of the Green and Robin functions.
When combined with the present representation of the Neumann Green function in annular-type domains, this framework makes it possible to study multi-bubble configurations in the four-dimensional annulus by following the same conceptual strategy adopted by Salazar in dimension $N=3$. In this sense, the present results supply the analytical ingredient needed to handle the geometry of the annulus, while the work of Pistoia, Rago and Vaira gives the perturbative and variational tools required to carry out the finite-dimensional analysis in dimension $N=4$.
\end{remark}

\section*{Acknowledgments}
The author thanks Professor Giusi Vaira for her valuable advice and for the useful discussions.
\section*{Data Availability Statements}
All data generated or analysed during this study are included in this article.
\section*{Declarations}
{\bf Conflicts of Interest:} The authors declare they have no financial interests.

\end{document}